[1994/12/01]
\documentclass{amsart}

\usepackage{amsfonts,amssymb,color,graphicx}
\usepackage{mathrsfs}
\usepackage{pdfsync,enumitem}
\usepackage{amscd}

\usepackage[active]{srcltx}
\usepackage[all]{xy}
\usepackage{calrsfs}
\usepackage{times}
\usepackage[scaled=0.92]{helvet}
\usepackage[usenames,dvipsnames,svgnames,table]{xcolor}

\usepackage{hyperref}
\hypersetup{
    colorlinks=true,       
    linkcolor=blue,          
    citecolor=blue,        
    filecolor=blue,      
    urlcolor=blue           
}

\usepackage{enumitem}

\chardef\bslash=`\\ 

\newtheorem{theorem}{Theorem}
\newtheorem*{theorem*}{Theorem}

\newtheorem{proposition}{Proposition}
\newtheorem{corollary}{Corollary}

\theoremstyle{definition}

\newtheorem{remark}{Remark}
\newtheorem{example}{Example}

\newcommand{\eval}[2][\right]{\relax
  \ifx#1\right\relax \left.\fi#2#1\rvert}

\usepackage{xcolor}

\begin{document}

\title{Automorphisms of Generalized Fermat manifolds}


\author[R. A. Hidalgo]{Rub\'en A. Hidalgo}
\address{Departamento de Matem\'atica y Estad\'istica, Universidad de la Frontera,  Temuco, Chile}
\email{ruben.hidalgo@ufrontera.cl}

\author[H. F. Hughes]{Henry F. Hughes}
\address{Instituto de Ciencias F\'{\i}sicas y Matem\'aticas, Facultad de Ciencias, Universidad Austral de Chile\\  Valdivia\\ Chile. Departamento de Matem\'atica y Estad\'istica, Universidad de la Frontera\\ Casilla 54-D, Temuco \\Chile}
\email{henry.hughes@uach.cl}

\author[M. Leyton-\'Alvarez]{Maximiliano Leyton-\'Alvarez}
\address{Instituto de Matem\'atica y F\'isica, Universidad de Talca\\
3460000 Talca, Chile}
\email{leyton@inst-mat.utalca.cl}

\thanks{Partially supported by the projects Fondecyt 1230001, 1220261, and 1221535}


\begin{abstract}
Let $d \geq 1$, $k \geq 2$ and $n\geq d+1$ be integers. A $d$-dimensional smooth complex algebraic variety
$M$ is called a generalized Fermat variety of type $(d;k,n)$ if there is a Galois holomorphic branched covering $\pi:M \to {\mathbb P}^{d}$, with deck group 
$H\cong {\mathbb Z}_{k}^{n}$, whose branch divisor consists of $n+1$ hyperplanes in general position, each one of branch order $k$. In this case, $H$ is called a generalized Fermat group of type $(d;k,n)$. In previous work, we proved that the generalized Fermat group $H$ is unique in the following cases: (i) $d=1$ and $(k-1)(n-1)>2$, or (ii) $d \geq 2$ and $(d;k,n) \notin \{(2;2,5), (2;4,3)\}$. To obtain this uniqueness fact, we used a differential method due to Kontogeorgis. This paper provides a different and shorter proof of the uniqueness of $H$. We also study the locus of fixed points of subgroups of $H$.
\end{abstract}

\maketitle

\section{Introduction}\label{sec:Int}
Let $M \subset {\mathbb P}^{n}$ be a $d$-dimensional smooth complex algebraic variety. In particular, as $M$ carries a natural structure of complex manifold, we may consider the group ${\rm Aut}(M)$ of its holomorphic automorphisms. We denote by ${\rm Lin}(M)$ its subgroup consisting of those automorphisms that are restrictions of elements of ${\rm PGL}_{n+1}({\mathbb C})$.

If $k,n \geq 2$ are integers, then we say that a subgroup $H \leq {\rm Aut}(M)$ is a {\it generalized Fermat group} of type $(d;k,n)$ if (i) $H \cong {\mathbb Z}_{k}^{n}$ and (ii) $M/H$ is the orbifold whose underlying space is the $d$-dimensional projective space ${\mathbb P}^{d}$ and whose branch locus is the complete intersection of $n+1$ hyperplanes in general position, and each one with branch order $k$. In this case, we also say that $M$ (respectively, $(M,H)$) is a {\it generalized Fermat manifold} (respectively, a {\it generalized Fermat pair}) of type $(d;k,n)$. 
Necessarily, $n \geq d$ (this follows from the fact that $M$ is smooth), and if $n=d$, then $M \cong {\mathbb P}^{d}$ (see Theorem \ref{nbajo}).

Let $(M,H)$ be a generalized Fermat pair of type $(d;k,n)$, where $n \geq d+1$.
Let us consider a Galois (branched) cover map $\pi:M \to {\mathbb P}^{d}$, with deck group $H$, and whose 
branch locus (the image of those points $x \in M$ with a non-trivial $H$-stabilizer; in which case is a cyclic group of order $k$) consists of the $n+1$ hyperplanes $\Sigma_{1},\ldots, \Sigma_{n+1}$, which are in general position. 
As a consequence of the classification of abelian coverings, due to Pardini in \cite{Pardini}, the triple $(M,H,\pi)$ is completely determined by the above hyperplanes. 

If  $n=d+1$, then (up to isomorphisms) $M=M^{k}_{d+1}$ is the Fermat hypersurface of degree $k$, for which ${\rm Lin}(M)=H \rtimes {\mathfrak S}_{n+1}$, where the permutation part acts as a permutation of coordinates (in particular, $H$ is the unique generalized Fermat group of type $(d;k,d+1)$).
If $n \geq d+2$, then an explicit algebraic model of $M$, for which $H$ is given as a very simple linear group of automorphisms, is given as a complete intersection of $(n-d)$ Fermat hypersurfaces of dimension $d$ and of degree $k$ in ${\mathbb P}^{n}$ (see Section \ref{Sec:2} and also \cite{Gao}). 
For $d=2$, such an algebraic model was already known to Hirzebruch in \cite{Hirzebruch1} (in \cite{Hirzebruch2}, Hirzebruch studied the arrangement of $n+1$ lines in ${\mathbb P}^{2}$ which are not necessarily in general position). In \cite{Bin}, Bin studied some examples (case $k=2$ and $n=4$) and their corresponding Pardini's building data. 

By resuts due to Kontogeorgis \cite{Kon02}, if $(d;k,n) \notin \{(2;2,5), (2;4,3)\}$, then ${\rm Aut}(M)={\rm Lin}(M)$. The two exceptional tuples correspond to the unique generalized Fermat surfaces being K3 surfaces. 
If $(d;k,n)=(2;4,3)$, then $M$ corresponds to the classical Fermat hypersurface of degree $4$ in ${\mathbb P}^{3}$ for which ${\rm Lin}(M) \cong {\mathbb Z}_{4}^{3} \rtimes {\mathfrak S}_{4}$ and ${\rm Aut}(M)$ infinite. If $(d;k,n)=(2;2,5)$, then ${\rm Lin}(M)$ is a finite extension of ${\mathbb Z}_{2}^{5}$ (generically a trivial extension) and ${\rm Aut}(M)$ is infinite by results due to Shioda and Inose in \cite[Thm 5]{Shioda} (in \cite{Vinberg} Vinberg computed it for a particular case).
In these exceptional cases, the generalized Fermat group $H$ (which is unique in ${\rm Lin}(M)$) cannot be a normal subgroup of ${\rm Aut}(M)$. Otherwise, every element of ${\rm Aut}(M)$ will induce an automorphism of ${\mathbb P}^{2}$ (permuting the $n+1$ branched lines), so a finite group of linear automorphism. This will ensure that ${\rm Aut}(M)$ is a finite extension of $H$, so a finite group, a contradiction.

\smallskip

In \cite{HKLP17}, it was proved that, for $d=1$ and $(n-1)(k-1)>2$, $M$ has a unique generalized Fermat group of type $(1;k,n)$. In \cite{HHL}, by applying a differential method due to Kontogeorgis \cite{Kon02}, we proved that, for $d \geq 2$, the group ${\rm Lin}(M)$ admits a unique generalized Fermat group of type $(d;k,n)$ (so, in the non-exceptional cases, the uniqueness in ${\rm Aut}(M)$). In this paper, we provide a different argument by an inductive process to obtain the uniqueness. One of the interests of this uniqueness result is that it provides a natural short exact sequence $1 \to H \to {\rm Aut}(M)={\rm Lin}(M) \to {\rm Aut}_{orb}(M/H) \to 1$, where ${\rm Aut}_{orb}(M/H)$ is the subgroup of the group ${\rm PGL}_{d+1}({\mathbb C})$ of conformal automorphisms of the ${\mathbb P}^{d}$ that keeps invariant the $(n+1)$ branch hyperplanes. This could be used to compute ${\rm Aut}(M)$ (see \cite{GHL09} for the case $d=1$).

We also describe the locus of fixed points of subgroups of $H$ (Proposition \ref{observafijos}), and we compute the plurigenera and the arithmetic genus of $M$
(Proposition \ref{cohomologia}). Moreover, we provide the arithmetic conditions of the tuple $(d;k,n)$ for $M$ to be a Calabi-Yau variety (in particular, if $d=2$, to be K3).

In Section \ref{Ssec:Aut}, we obtain information of ${\rm Aut}(M)$. For instance, by the Schur-Zaseenhaus theorem, we obtain that if $k$ is relatively prime to the ${\rm PGL}_{d+1}({\mathbb C})$-stabilizer $G_{0}$ of the branch divisor of $M/H$, then ${\rm Aut}(M) \cong H \rtimes G_{0}$.

In the last section, we provide some examples.

\smallskip
{\bf Notations:}
\begin{enumerate}[leftmargin=*,align=left]
\item If $n \geq 1$ and $k \geq 2$ are integers, then ${\mathbb Z}_{k}^{n}:={\mathbb Z}_{k} \times \stackrel{n}{\cdots} \times {\mathbb Z}_{k}$, where ${\mathbb Z}_{k}$ denotes the cyclic group of order $k$.
\item We denote by $M^{k}_{d+1}$ the degree $k$ Fermat hypersurface in ${\mathbb P}^{d}$. For $d=1$ (resp., $d=2$) we also use the notation $C^{k}_{2}$ (resp., $S^{k}_{3}$).

\item The divisor $D$ given by the union of $n+1$ hyperplanes $L_{1},\ldots,L_{n+1} \subset {\mathbb P}^{d}$, where $n \geq 2$, which are in general position, is called a strict normal crossing divisor.

\item For $n \geq d+1$ and  a strict normal crossing divisor $D=L_{1}+\cdots+L_{n+1}$ (where, for $d=1$, $D$  consists of $n+1$ different points of ${\mathbb P}^{1}$), we will use the notation $M^{k}_{n}(D)$  for a generalized Fermat manifold of type $(d;k,n)$ associated to $D$. In the special case $d=1$ (resp., $d=2$), we will also use the notation $C^{k}_{n}(D)$ (resp., $S^{k}_{n}(D)$) for a generalized Fermat manifold of type $(1;k,n)$ (resp., $(2;k,n)$).  In  \cite{GHL09}, it is proven that $C^{k}_{n}(D)$ is uniquely determined by the isomorphism class of $D$ (class defined by the automorphisms of ${\mathbb P}^{1}$).
\end{enumerate}

\section{Generalized Fermat manifolds}\label{Sec:2}
In this section, we provide suitable algebraic models for generalized Fermat manifolds of type $(d;k,n)$, obtained as a fiber product of $(n-d)$ classical Fermat hypersurfaces of degree $k$. 
\subsection{Hyperplanes in general position}
Let's start by recalling some basic facts on hyperplanes in general position as it is important in defining generalized Fermat pairs. First, note that each hyperplane in $L \subset {\mathbb P}^{d}$ has the form $L_{q}:=\{\rho_{1}t_{1}+\cdots+\rho_{d+1}t_{d+1}=0\}$. where $q:=[\rho_{1}:\cdots:\rho_{d+1}] \in {\mathbb P}^{d}$.

A collection $L_{q_{1}},\ldots,L_{q_{n+1}} \subset {\mathbb P}^{d}$, where $n \geq d+1$, of hyperplanes are in {\it general position} if
the corresponding (pairwise different) points $q_{1},\ldots, q_{n+1} \in {\mathbb P}^{d}$ are in general position. This means that, for $d \geq 2$, any subset of $3 \leq s \leq d+1$ of these points spans a $(s-1)$-plane $\Sigma \subset {\mathbb P}^{d}$.
In this above situation, the divisor $D$, formed by the hyperplanes  $L_{q_{1}},\ldots,L_{q_{n+1}}$, is a strict normal crossing divisor.
If $q_{j}=[\rho_{1,j}:\cdots:\rho_{d+1,j}] \in {\mathbb P}^{d}$, for $j=1,\ldots,n+1$, then
the hyperplanes $L_{q_{1}},\ldots,L_{q_{n+1}}$ are in general position if and only if, for every all $(d+1 \times d+1)$-minors of the matrix
{\small
$$
M(q_{1},\ldots,q_{n+1}):=
\left(\begin{array}{ccc}
\rho_{1,1} & \cdots & \rho_{1,n+1}\\
\vdots & \vdots & \vdots \\
\rho_{d+1,1} & \cdots & \rho_{d+1,n+1}
\end{array}
\right)  \in {\rm M}_{(d+1)\times(n+1)}({\mathbb C})
$$
}
are nonzero.
For example, if $e_{1}=[1:0:\cdots:0], e_{2}=[0:1:0:\cdots:0],\ldots, e_{d+1}=[0:\cdots:0:1]$ and $e_{d+2}=[1:\cdots:1]$, then
the hyperplanes $L_{e_{1}}, \ldots, L_{e_{d+2}}$
are in a general position.

Let $n \geq d+1$. Two ordered tuples $(L_{q_{1}},\ldots,L_{q_{n+1}})$ and $(L_{p_{1}},\ldots,L_{p_{n+1}})$, of hyperplanes in ${\mathbb P}^{d}$ in general position, are equivalent if there is a $T \in {\rm PGL}_{d+1}({\mathbb C})$ such that $L_{p_{j}}=T(L_{q_{j}})$, for $j=1,\ldots, n+1$.
We denote by $X_{n,d}$ the set of such equivalence classes.

If  $(L_{q_{1}},\ldots,L_{q_{n+1}})$ is one of such tuples, then there is a unique $T \in {\rm PGL}_{d+1}({\mathbb C})$ such that
$T(L_{q_{j}})=L_{e_{j}}$, for $j=1,\ldots, d+2$. Then, for each $j=d+3,\ldots,n+1$,  there is some unique
$\Lambda_{i}:=[\lambda_{i,1}:\cdots:\lambda_{i,d}:1] \in {\mathbb P}^{d}, \quad i=1,\ldots,n-d-1$, such that 
$T(L_{q_{j}})= L_{\Lambda_{j-d-2}}$. In this case, we set 
$\Lambda:=(\lambda_{1},\ldots,\lambda_{d}) \in {\mathbb C}^{d(n-d-1)}$, where 
$\lambda_{j}:=(\lambda_{1,j},\ldots,\lambda_{n-d-1,j}) \in {\mathbb C}^{n-d-1}$,
for $j=1,\ldots, d$. So, if we set 
{\small
$$\begin{array}{l}
L_{j}(\Lambda):= L_{e_{j}} \subset {\mathbb P}^{d}, \; j=1,\ldots, d+2,\\
L_{j}(\Lambda):= L_{\Lambda_{j-d-2}} \subset {\mathbb P}^{d}, \; j=d+3,\ldots,n+1,
\end{array}
$$
}
then
$(L_{1}(\Lambda), \ldots, L_{n+1}(\Lambda))$ is equivalent to $(L_{q_{1}},\ldots,L_{q_{n+1}})$. We call the tuple $\Lambda$ a standard parameter.

This observation permits us to identify 
(i) $X_{d+1,d}$ with the one set-point $\{(1,\stackrel{d}{\ldots},1)\}$ and (ii) for $n \geq d+2$, $X_{n,d}$ with the set of tuples $\Lambda \in {\mathbb C}^{d(n-d-1)}$ such that the $n+1$ hyperplanes $L_{1}(\Lambda),\ldots,L_{n+1}(\Lambda)$ are in a general position. We observe that in case (ii), $X_{n,d}$ is a (non-empty) open set.

\begin{remark}[Moduli spaces of generalized Fermat manifolds]
The symmetric group ${\mathfrak S}_{n+1}$ acts by permuting the $n+1$ coordinates of an ordered tuple of $n+1$ hyperplanes in general position. Such an action induces an action  
of a group ${\mathbb G}_{n,d}$ of holomorphic automorphisms on $X_{n,d}$. The quotient complex orbifold $\overline{X}_{n,d}=X_{n,d}/{\mathbb G}_{n,d}$ provides a parameter space of (unordered) collection of $n+1$ hyperplanes in general position. As a consequence of the uniqueness of the generalized Fermat groups of type $(d;k,n) \notin \{(2;2,5),(2;4,3)\}$, the orbifold $\overline{X}_{n,d}$ can be identified with the moduli space of generalized Fermat manifolds of type $(d;k,n)$.
\end{remark}

\subsection{Algebraic models of generalized Fermat pairs for ${\bf n \geq d+1}$}\label{ssec:Ald-des}
Let us consider integers $n \geq d+1$, $k \geq2 $ and $d \geq 1$. Let $(M,H)$ be a generalized Fermat pair of type $(d;k,n)$ and $\pi:M \to {\mathbb P}^{d}$ be an abelian branched covering with deck group $H$. Up to post-composition by some linear automorphism of ${\mathbb P}^{d}$, we may assume that its branch locus consists of the hyperplanes $L_{1}(\Lambda), \ldots, L_{n+1}(\Lambda)$, 
$\Lambda:(\lambda_1,...,\lambda_d)\in X_{n,d}$, where
$\lambda_{j}:=(\lambda_{1,j},\ldots,\lambda_{n-d-1,j}) \in {\mathbb C}^{n-d-1}$.

\subsubsection{\bf An algebraic variety $M^{k}_{n}(\Lambda)$}
Let us consider the following complex projective algebraic set
{\small
\begin{equation} \label{GFM}
M^{k}_{n}(\Lambda):=\left \{ \begin{array}{rcc}
              x_{1}^{k}+\cdots + x_{d}^{k}+ x_{d+1}^{k}+ x_{d+2}^{k}&=&0\\
              \lambda_{1,1}x_{1}^{k}+\cdots +\lambda_{1,d}x_{d}^{k}+x_{d+1}^{k}+x_{d+3}^{k}&=&0\\
              \vdots\hspace{1cm} &\vdots &\vdots\\
              \lambda_{n-d-1,1}x_{1}^{k}+\cdots +\lambda_{n-d-1,d}x_{d}^{k}+x_{d+1}^{k}+x_{n+1}^{k}&=&0\\
             \end{array}\right \}\subset {\mathbb P}^{n}.
\end{equation}
}

\begin{remark}
If $n=d+1$, then we will use the notations
$C_{2}^{k}=\{x_{1}^{k}+x_{2}^{k}+x_{3}^{k}=0\} \subset {\mathbb P}^{2}$,
$S^{k}_{3}:=\{x_{1}^{k}+x_{2}^{k}+x_{3}^{k}+x_{4}^{k}=0\} \subset {\mathbb P}^{3}$ and $M^{k}_{d+1}:=\{x_{1}^{k}+x_{2}^{k}+\cdots+x_{d+2}^{k}=0\} \subset {\mathbb P}^{d+1}$.
\end{remark}

\begin{proposition}\label{sing-set}
If $\Lambda \in X_{n,d}$, then $M^{k}_{n}(\Lambda)$ is an irreducible nonsingular complete intersection.
\end{proposition}
\begin{proof}
Let us observe that the matrix of coefficients of \eqref{GFM} has rank $(m+1)\times(m+1)$ and all of its $(m+1)\times(m+1)$-minors are different from zero (this is the general position condition of the $n+1$ hyperplanes). The result follows from \cite[Proposition 3.1.2]{Ter88}.  For explicitness, let us work out the case $d=2$ (the case $d=1$ was noted in \cite{GHL09}).

Set $\lambda_{0,1}=\lambda_{0,2}=1$ and consider the following degree $k$ homogeneous polynomials
$f_{i}:=\lambda_{i,1} x_{1}^{k}+\lambda_{i,2}x_{2}^{k}+x_{3}^{k}+x_{4+i}^{k} \in {\mathbb C}[x_{1},\ldots,x_{n+1}]$, where $i \in \{0,1,\ldots,n-3\}$. Let $V_{i} \subset {\mathbb P}^{n}$ be the hypersurface given as the zero locus of $f_{i}$.

The algebraic set $S^{k}_{n}(\lambda_{1},\lambda_{2})$ is the intersection of the $(n-2)$ hypersurfaces $V_{0}, \ldots, V_{n-3}$.
We consider the matrix of  $\nabla f_i$ written as rows.
{\small
\begin{equation} \label{Jac}
\begin{pmatrix}
k x_{1}^{k-1} & k x_{2}^{k-1} &  k x_{3}^{k-1} & k x_{4}^{k-1}& 0& \ldots & 0 \\
\lambda_{1,1} k x_{1}^{k-1} & \lambda_{1,2} k x_{2}^{k-1} &  k x_{3}^{k-1}& 0 & k x_{5}^{k-1}& \ldots &  0 \\
 \vdots      & \vdots    & \vdots  & \vdots  & \vdots \\
 \lambda_{n-3,1} k x_{1}^{k-1} & \lambda_{n-3,2} k x_{2}^{k-1} &  k x_{3}^{k-1} & 0 & \ldots &  0 &
  kx_{n+1}^{k-1}
\end{pmatrix}.
\end{equation}
}
By the  defining equations of the curve, and as $(\lambda_{1},\lambda_{2}) \in X_{n}$,  we see that a point
which has three variables $x_{i}=x_{j}=x_{l}=0$ for  $i\neq j \neq l \neq i$ and $1\leq i,j,l \leq n+1$ has also  $x_t=0$ for $t=1,\ldots,n+1$. Therefore, the above matrix has the maximal rank $n-2$ at all points of the surface.  So, the defining hypersurfaces are intersecting transversally, and
the corresponding algebraic surface they define is a nonsingular complete intersection.

Next, we proceed to see that the ideal $I_{k,n}$,  defined by the $n-2$ equations defining $S^{k}_{n}(\lambda_{1},\lambda_{2}) \subset \mathbb{P}^{n}$ is prime. This follows similarly as in \cite[sec. 3.2.1]{Kon02}. Observe
first that the defining equations $f_{0},\ldots,f_{n-3}$ form a regular sequence,
that ${\mathbb C}[x_{1},\ldots,x_{n+1}]$ is a Cohen-Macauley ring and that the ideal $I_{k,n}$ they define
is of codimension $n-2$. The ideal $I_{k,n}$ is prime as a consequence of the Jacobian Criterion \cite[Th. 18.15]{Eis95},
\cite[Th. 3.1]{Kon02} and that the fact that the matrix \eqref{Jac} is of maximal rank $n-2$ on $S^{k}_{n}(\lambda_{1},\lambda_{2})$. In remark \cite[3.4]{Kon02}, it is pointed out that an ideal $I$ is prime if the singular locus of the algebraic set defined by $I$  has big enough codimension (in our case, the singular set is the empty set).
\end{proof}

\begin{remark} \label{rem:sc}
If $d \geq 2$, then (as $M^{k}_{n}(\Lambda)$ is an irreducible nonsingular complete intersection projective variety of dimension $d$) it follows that $M^{k}_{n}(\Lambda)$ is simply connected (this result is attributed to Lefschetz. For more information, see \cite{Har74}).
\end{remark}

\subsubsection{\bf The group $H_{0}$}
Let us now consider the group $H_{0}=\langle\varphi_{1},\ldots,\varphi_{n+1} \rangle$, where
\begin{center}
 $\varphi_j([x_{1}:\cdots:x_{j}:\cdots:x_{n+1}]):=[x_{1}:\cdots:w_{k}x_{j}:\cdots:x_{n+1}]$,\end{center}
and $w_k$ is a primitive $k$-th root of unity.
Let ${\rm Fix}(\varphi_{j}) \subset M^{k}_{n}(\Lambda)$ be the set of fixed points of  $\varphi_{j}$ and $F(H_{0}):=\cup_{j=1}^{n+1}{\rm Fix}(\varphi_{j})$.
The following facts can be deduced from the above.
\begin{enumerate}[label=(\alph*),leftmargin=*,align=left]
\item[(I)] $H_{0} \cong {\mathbb Z}_{k}^{n}$.
\item[(II)] $\varphi_{1}  \varphi_{2} \cdots  \varphi_{n+1}=1$.
\item[(III)]$H_{0}<{\rm Aut}(M^{k}_{n}(\Lambda))<{\rm PGL}_{n+1}({\mathbb C})$.
\item[(IV)] The only non-trivial elements of $H_{0}$ with fixed set points being of maximal dimension $d-1$ are the non-trivial powers
of the generators $\varphi_{1}, \ldots, \varphi_{n+1}$ (see Proposition \ref{observafijos}). Moreover, ${\rm Fix}(\varphi_{j}):=\{x_{j}=0\} \cap M^{k}_{n}(\Lambda)$,
which is isomorphic to a generalized Fermat manifold of type $(d-1;k,n-1)$ for $d \geq 2$ and a collection of $k^{n-1}$ points for $d=1$.
We call the set $\varphi_{1},\ldots, \varphi_{n+1}$ a canonical generators of $H_{0}$ (canonical generators are not unique, but any other set of them is of the form $\varphi_{1}^{r},\ldots,\varphi_{n+1}^{r}$, where $r \in \{1,\ldots,k-1\}$ is relatively prime to $k$).
\end{enumerate}

\subsubsection{\bf An algebraic models of $(M,H)$}

\begin{theorem}
$(M^{k}_{n}(\Lambda),H_{0})$ is a generalized Fermat pair of type $(d;k,n)$ and there is a biholomorphism between $M$ and $M^{k}_{n}(\Lambda)$ conjugating $H$ to $H_{0}$. In particular, every generalized Fermat manifold of type $(d;k,n)$ whose associated branching divisor is ${\rm PGL}_{d+1}({\mathbb C})$ equivalent to the divisor associated to $\Lambda$ is isomorphic to $M^{k}_{n}(\Lambda)$.
\end{theorem}
\begin{proof}
The map
{\small
\begin{equation} \label{pi0}
\pi_{0}:M^{k}_{n}(\Lambda) \to {\mathbb P}^{d}: [x_{1}:\cdots:x_{n+1}] \mapsto [x_{1}^{k}: \cdots: x_{d+1}^{k}]
\end{equation}
}
is a regular branched cover with deck group $H_{0}$ and whose branch set is the union of the previous hyperplanes $L_{1}(\Lambda),\ldots,L_{n+1}(\Lambda)$, each one of order $k$. In other words, the pair $(M^{k}_{n}(\Lambda),H_{0})$ is a generalized Fermat pair of type $(d;k,n)$. The last part is a consequence of Pardini's building data.
\end{proof}

\begin{remark}\label{productofibrado}
\begin{enumerate}[leftmargin=*,align=left]
\item If $\Lambda=(\lambda_{1},\ldots,\lambda_{d}) \in X_{n,d}$, where $n \geq d+1$, and setting
$\lambda_{0,i}=1$, for $i=1,\ldots,d$ and, for each $j \in \{0,1,\ldots,n-d-1\}$, we consider the classical degree $k$ Fermat hypersurface
$F_{j}=\{\lambda_{j,1}x_{1}^{k}+\cdots+\lambda_{j,d}x_{d}^{k}+x_{d+1}^{k}+x_{d+2+j}^{k}=0\} \subset {\mathbb P}^{d+1}$ and the rational map $\pi_{j}:F_{j} \to {\mathbb P}^{d}$ defined by $\pi_{j}([x_{1}:\cdots:x_{d+1}:x_{d+2+j}])=[x_{1}^{k}:\cdots:x_{d+1}^{k}]$. The branch values of $\pi_{j}$ are given by the hyperplanes
$L_{i}(\Lambda)$, for $i=1,\ldots,d+1$, and $L_{d+2+j}(\Lambda)$. If we consider the fiber product of all these pairs $(F_{0},\pi_{0}), \ldots, (F_{n-d-1},\pi_{n-d-1})$, then we obtain a reducible projective algebraic variety with the action of the group ${\mathbb Z}_{k}^{n(n-d)}$ and quotient ${\mathbb P}^{d}$ with branching divisor given by the previous $n+1$ hyperplanes. This fiber product has $k^{(n-d)(d+1)-n}$ irreducible components, each of them  isomorphic to $M^{k}_{n}(\Lambda)$.

\item (An inductive process)
If we set $H_{j}:=\langle \varphi_{1},\ldots, \varphi_{j-1}, \varphi_{j+1},\ldots, \varphi_{n}\rangle$, where $1 \leq j \leq n$, and  $H_{n+1}:=\langle \varphi_{1},\ldots,\varphi_{n-2},\varphi_{n+1}\rangle$, then $H_{j} \cong {\mathbb Z}_{k}^{n-1}$.
For $d \geq 2$, $H_{j}<{\rm Aut}(F_{j})$ and there is a holomorphic map $\pi_{j}:F_{j} \to L_{j}(\Lambda)={\mathbb P}^{d-1}$ (given by the restriction of $\pi_{0}$), which is a Galois holomorphic branched cover with deck group $H_{j}$, such that its branch locus is given by the collection of the $n$ intersections $\pi(F_{j}) \cap \pi(F_{i})$, $i \neq j$, which are copies of ${\mathbb P}^{d-2}$ in general position.  In particular, $(F_{j},H_{j})$ is a generalized Fermat pair of type $(d-1;k,n-1)$. This permits the study of generalized Fermat manifolds from an inductive point of view.

\item 
In \cite{Ter88}, T. Terasoma studied a family (depending on $n+1$ parameters) $X_{r,n,k} \subset {\mathbb P}^{n}$, being a complete intersection of $r+1$ Fermat hypersurfaces of degree $k \geq 2$, where $r+1<n$, and obtained a decomposition of its Hodge structure $H^{n-r-1}_{prim}(X_{r,n,k},{\mathbb Q})$. These algebraic varieties admit a group of automorphisms $H \cong {\mathbb Z}_{k}^{n}$ such that $X/H$ is ${\mathbb P}^{n-r-1}$, and its branch locus is given by exactly $n+1$ linear hyperplanes in general position. It follows that $X_{r,n,k}$ is an example of a generalized Fermat manifold of type $(n-r-1;k,n)$. Now, for instance, if we take
$r=n-3$, the surface $X_{n-3,n,k}$ is an example of a generalized Fermat surface of type $(2;k,n)$. As generalized Fermat surfaces of type $(2;k,n)$ have moduli dimension $2n-6$, for $n>7$, Terasoma's examples form a strict subfamily.
\end{enumerate}

\end{remark}

\subsection{Algebraic model of a generalized Fermat pair for ${\bf n=d}$}\label{Sec:excepcional}
In the above, we have assumed the condition $n \geq d+1$. When $2 \leq n = d$, then we have the following fact.

\begin{theorem}\label{nbajo}
If $(M,H)$ is a generalized Fermat pair of type $(d;k,d)$, where $d,k \geq 2$, then  (up to isomorphisms) 
$M={\mathbb P}^{d}$ and $H=\langle \varrho_{1},\ldots,\varrho_{d} \rangle$, where $\varrho_{j}([x_{1}:\cdots:x_{d+1}])=[x_{1}:\cdots:x_{j-1}:\omega_{k} x_{j}: x_{j+1}:\cdots:x_{d+1}]$.
\end{theorem}
\begin{proof}
Let us consider a Galois branched covering $\pi:M \to {\mathbb P}^{d}$, with deck group $H$, and whose branch divisor locus is given by the hyperplanes $L_{1},\ldots,L_{n+1}$. Up to post-composition of $\pi$ by a suitable projective linear automorphisms, we may assume $L_{j}=\{t_{j}=0\}$. Now, we add the hyperplane $L_{d+2}$, where $L_{d+2}=\{t_{1}+\cdots+t_{d+1}=0\}$. Next, consider the orbifold $M^{orb}$, whose underlying complex manifold is $M$ and whose orbifold divisor is obtained by lifting the hyperplane $L_{d+2}$ and giving to each component of its lifting the branching order equal to $k$. By construction, the group $H$ preserves that branch divisor of $M^{orb}$.
Let $M^{k}_{d+1} \subset {\mathbb P}^{d+1}$ be the generalized Fermat manifold of type $(d;k,d+1)$ (the Fermat hypersurface of degree $k$)  together its generalized Fermat group $H_{0}=\langle \varphi_{1},\ldots,\varphi_{d+1}\rangle$ and the corresponding Galois branched covering $\pi_{0}:M_{d+1}^{k} \to {\mathbb P}^{d}$ with deck group $H_{0}$.
Note that $M^{orb}/H= M_{d+1}^{k}/H_{0}$ (which is ${\mathbb P}^{d}$ with branch divisor $L_{1}+\cdots+L_{d+2}$, each one with order $k$).
As the universal covering of this last orbifold is $M^{k}_{d+1} \subset {\mathbb P}^{d+1}$, the branched covering $\pi_{0}$  it must factors through $\pi$, that is, there is some ${\mathbb Z}_{k} \cong K<H_{0}$ such that $M=M_{d+1}^{k}/K$.
Because of the choices of the hyperplanes $L_{j}$, we must have $K=\langle \varphi_{d+2}\rangle,$
where $\varphi_{1},\ldots, \varphi_{d+2} \in H_{0}$ are the corresponding standard set of generators.
 We may consider the (branched) Galois cover, with deck group $K$,
 {\small
$$\pi_{K}:M_{d+1}^{k} \to {\mathbb P}^{d}: [x_{1}:\cdots:x_{d+2}] \mapsto [x_{1}:\cdots:x_{d+1}]=[y_{1}:\cdots:y_{d+1}].$$
}
In this way, $M \cong {\mathbb P}^{d}$ and $H=\langle \varrho_{1},\ldots,\varrho_{d}\rangle$, where $\varrho_{j}$ is as decribed in the theorem.
The corresponding Galois branched cover, with deck group $H$, is
$\pi([y_{1}:\cdots:y_{d+1}])=[y_{1}^{k}:\cdots:y_{d+1}^{k}]$ whose branch divisor is $L_{1}+\cdots+L_{d+1}$.
\end{proof}

\subsection{A remark on the cohomological information of generalized Fermat manifolds}\label{re:cohom}
The fact that $M:=M^{k}_{n}(\Lambda)$ is a complete intersection variety allows us to compute the cohomology groups of the twisting sheaf $\mathcal{O}_{M}(r)$ in a relatively direct way, and in particular, to obtain the following.

\begin{proposition}\label{cohomologia}
Let $\Lambda \in X_{n,d}$, $n \geq d+1$, and $M:=M^{k}_{n}(\Lambda)$. Then
\begin{enumerate}[leftmargin=*,align=left]
\item The plurigenera $P_{m}(M)$ of $M$ satisfies
{\small
$$P_{m}(M)= \frac{k^{n-d}((n-d)k-n-1)^{d}}{d!}m^{d}+O(m^{d-1}).$$
}

\item The arithmetic genus $p_{a}(M)$ and the geometric genus $p_{g}(M)$ are given by
{\small
$$p_a(M)=p_g(M)= \left \{ \begin{array}{ccc}
0 & \mbox{if} & r_1<0\\
            \binom{r_1+n}{n} & \mbox{if} & 0\leq r_1<k\\
             & &\\
            \sum_{j\in \Delta_{r_1}}\binom{r_1-\overline{j}+d}{d} & \mbox{if} & r_1\geq k\\
           \end{array} \right .$$
           }

\item If $(n-d)k-n-1=0$, then $M$ is a Calabi-Yau variety.
\item If  $d=2$, then $M$ is a general type surface except for the rational varieties cases $(k,n)\in \{(2,3), (3,3), (2,4)\}$ and the $K3$ varieties $(k,n)\in \{(4,3), (2,5)\}$.

\end{enumerate}
\end{proposition}
\begin{proof}
Let $\mathbb{C}[x_1,...,x_{m}]_l$ be  the  homogeneous polynomials of degree $l$.

\begin{enumerate} [label=(\alph*),leftmargin=*,align=left]
\item We first proceed to describe the cohomology groups of the twisting sheaf   $\mathcal{O}_{M}(r), r\in \mathbb{Z}.$
\begin{enumerate}[leftmargin=*,align=left]
 \item[(a1)]  Let  $\Delta_r:=\{(j_1,...,j_{n-d})\in \mathbb{Z}^{n-d}:\; 0\leq j_i\leq k-1, 0\leq i\leq n-d, \;\mbox{and}\; \overline{j}:=j_1+j_2+\cdots j_{n-d}\leq r \}$. Then
{\small
 $$\displaystyle H^{0}(M, \mathcal{O}_{M}(r)):=\left \{  \begin{array}{ccc}
                                             0 & \mbox{if} & r<0\\
                                           \mathbb{C}[x_1,...,x_{n+1}]_r &\mbox{if} & 0\leq r<k \\
                                           \bigoplus_{j\in \Delta_r}Q_j&\mbox{if}  & r\geq k
                                          \end{array} \right .$$
                                          }
where  $ Q_j:=\mathbb{C}[x_1,...,x_{d+1}]_{(r-\overline{j})}x_{d+2}^{j_1} x_{d+3}^{j_2}\cdots x_{n+1}^{j_{n-d}}$, $j:=(j_1,....j_{n-d})$.
\item[(a2)] By Grothendieck's vanishing theorem, we have that:
{\small
\begin{center}
 $H^{i}(M, \mathcal{O}_{M}(r))=0$  for $i>d$ and $r\in \mathbb{Z}$,
\end{center}
}

\item[(a3)]  and   as $M$ is a complete intersection variety we obtain that:
{\small
 \begin{center}
  $H^{i}(M, \mathcal{O}_{M}(r))=0$ for $0<i<d$ and $r\in \mathbb{Z}$,
 \end{center}
 }
 (see page 231 of \cite{Har77}).
 \item[(a4)]  Finally, using the Serre duality, we obtain that:
 {\small
  \begin{center}
     $H^{d}(M, \mathcal{O}_{M}(r))\cong H^{0}(M, \mathcal{O}_{M}(r_1-r))$
  \end{center}
  }

Remember that $\omega_M\cong  \mathcal{O}_M(r_1)$, $r_1=(n-d)k-n-1$, (see page 188 of \cite{Har77}).

\end{enumerate}

\item With the former, we can calculate the plurigenus of $M$
{\small
$$P_m(M)=\dim_{\mathbb{C}} H^{0}(M,\omega_M^{\otimes m})=\dim_{\mathbb{C}} H^{0}(M,\mathcal{O}_M(r_m))$$
}
where $r_m:=mr_1=m((n-d)k-n-1)$.
\begin{enumerate}
 \item[(b1)]  If  $ (n-d)k-n-1<0$, we obtain that $P_{m}(M)=0$. This implies that the Kodaira dimension of $M$ is  $\kappa(M)=-\infty$.
 \item[(b2)] If  $(n-d)k-n-1=0$, we obtain that $P_{m}(M)=1$. This implies that the Kodaira dimension of $M$ is   $\kappa(M)=0$.
 \item[(b3)] If $(n-d)k-n-1>0$, the canonical sheaf in very ample and
{\small
 \begin{center}
   $P_m(M)=\left \{ \begin{array}{ccc}

            \binom{r_m+n}{n} & \mbox{if} &0\leq  r_m<k\\
             & &\\
             \sum_{j\in \Delta_{r_m}}\binom{r_m-\overline{j}+d}{d} & \mbox{if} & r_m\geq k\\\
           \end{array} \right .
$
 \end{center}
 }
In particular, if $r_m\geq \max\{ k, (n-d)(k-1)\}$, we obtain that

{\small
$$P_m(M)= \frac{k^{n-d}((n-d)k-n-1)^{d}}{d!}m^{d}+O(m^{d-1})$$
}
\end{enumerate}

This implies that the Kodaira dimension of  $M$ is  $\kappa(M)=d$.

\item The former also permits us to determine the arithmetic genus and geometric genus of  $M$. As seen from the above,
{\small
$$p_a(M)=p_g(M)=\dim_{\mathbb{C}}H^{d}(M,\mathcal{O}_M)=\dim_{\mathbb{C}}H^{0}(M,\mathcal{O}_M(r_1)),$$
} so
{\small
$$p_a(M)=p_g(M)= \left \{ \begin{array}{ccc}
0 & \mbox{if} & r_1<0\\
            \binom{r_1+n}{n} & \mbox{if} & 0\leq r_1<k\\
             & &\\
            \sum_{j\in \Delta_{r_1}}\binom{r_1-\overline{j}+d}{d} & \mbox{if} & r_1\geq k\\
           \end{array} \right .$$
}
 \end{enumerate}

\end{proof}

\section{Uniqueness of the generalized Fermat group}\label{Sec:Aut}
In this section, we provide an alternative (and short) proof of the uniqueness of generalized Fermat groups to the one given in \cite{HHL}.

\begin{theorem}
Every generalized Fermat manifold $M$ of type $(d;k,n)$, where $n \geq d+2$, has a unique generalized Fermat group of the same type in ${\rm Lin}(M)$. In particular, if $(d;k,n) \notin \{(2;2,5), (2;4,3)\}$, then the uniqueness holds in ${\rm Aut}(M)$).
\end{theorem}
\begin{proof}
Let us assume $n \geq d+2$ and that $H<{\rm Lin}(M)$ is another generalized Fermat group of $M=M_{n}^{k}(\Lambda)$ of type $(d;k,n)$.
Let $\varphi_{1},\ldots, \varphi_{n+1}$ be an standard set of generators of $H_{0}$ and $\varphi^{*}_{1}, \ldots, \varphi^{*}_{n+1}$ be an standard set of generators for $H$.
Following in a similar way as done in the proof of Claim 1 of Section 5 in \cite{HKLP17}, we may assume that there are $i, j \in \{1,\ldots,n\}$ such that $\langle \varphi_{i} \rangle=\langle \varphi^{*}_{j} \rangle$. Up to a permutation of indices, we may assume $i=j=1$. So, let $\langle \varphi_{1} \rangle = \langle \varphi_{1}^{*} \rangle$ and set $F_{1} \subset M$ the locus of fixed points of $\varphi_{1}$, which is a generalized Fermat manifold of type $(d-1;k,n-1)$, there are two generalized Fermat groups of that type, these being $H_{0}/\langle \varphi_{1} \rangle$ and $H/\langle \varphi^{*}_{1} \rangle$. So, if we have the uniqueness for dimension $d-1$, then we will be done by an induction process.
The starting case is $d=2$, for which $F_{1}$ is a generalized Fermat curve of type $(k,n-1)$. The uniqueness, in this case, works if $(k-1)(n-2) >2$ \cite{HKLP17}. As we are assuming $n \geq d+2=4$, the uniqueness will hold if either (i) $n=4$ and $k\geq 3$ or (ii) $n \geq 5$. So, in these cases, $H_{0}=H$.

Let us now consider the left case $n=4$, $k=2$, and, by contradiction, assume $H \neq H_{0}$.
In this case, $F_{1}$ is a genus one Riemann surface. On it we have the group $J=H_{0}/\langle \varphi_{1} \rangle \cong {\mathbb Z}_{2}^{3}$ of conformal automorphisms.
If $j \neq 1$, then $F_{1}$ is tangent to each $F_{j}$ (the locus of fixed points of $\varphi_{j}$) at four points (these are the fixed points of the action of $\varphi_{j}$ on $F_{1}$).
The torus $F_{1}$ corresponds to the plane generalized Fermat curve of genus one
{\small
$$F_{1}=\left\{\begin{array}{rcl}
y_{1}^{2}+y_{2}^{2}+y_{3}^{2}&=&0\\
\mu y_{1}^{2}+y_{2}^{2}+y_{4}^{2}&=&0
\end{array}
\right\} \subset {\mathbb P}^{3}.
$$
}

In the above model, the group $J$ is generated by $a_{1}([y_{1}:y_{2}:y_{3}:y_{4}])=[-y_{1}:y_{2}:y_{3}:y_{4}]$, $a_{2}([y_{1}:y_{2}:y_{3}:y_{4}])=[y_{1}:-y_{2}:y_{3}:y_{4}]$ and
$a_{3}([y_{1}:y_{2}:y_{3}:y_{4}])=[y_{1}:y_{2}:-y_{3}:y_{4}]$, and the corresponding Galois branched covering is $\pi([y_{1};y_{2}:y_{3}:y_{4})=-(y_{2}/y_{1})^{2}=z$ whose branched values are $\infty$, $0$, $1$ and $\mu$. Set $a_{4}=a_{1}a_{2}a_{3}$.
As we have assumed above that $\varphi_{1}=\varphi^{*}_{1}$, then $F_{1}$
also admits the group $J'=H/\langle \varphi^{*}_{1} \rangle \cong {\mathbb Z}_{2}^{3}$ of conformal automorphisms; and it normalizes $J$. So, $J'$ induces (under $\pi$) a group of conformal automorphisms of the Riemann sphere that permutes the branch values of $\pi$ and isomorphic to ${\mathbb Z}_{2}^{t}$, for some $t>0$ (as we are assuming $H \neq H'$). There are only two possibilities: $t\in \{1,2\}$ (as the only finite abelian subgroups of M\"obius transformations are either cyclic or ${\mathbb Z}_{2}^{2}$). It follows that $J \cap J'$ can be either ${\mathbb Z}_{2}$ or ${\mathbb Z}_{2}^{2}$.
Since $H \neq H_{0}$, one of the standard generators of $H$ is different from the standard generators of $H_{0}$, say $\varphi^{*}_{2}$. This generator induces an involution $\alpha$ with four fixed points on $F_{1}$ different from those in $J$. We may assume that it induces the involution $a(z) = \mu/z$, in which case
$\alpha([y_{1}:y_{2}:y_{3}:y_{4}])=[y_{2}: \epsilon_{1} \sqrt{\mu} \; y_{1}: \epsilon_{2} y_{4}: \epsilon_{3} \sqrt{\mu} \; y_{3}]$, where $\sqrt{\mu}$ is a fixed square root of $\mu$,
$\epsilon_{j} \in \{\pm1\}$ and $\epsilon_{1}=\epsilon_{2}\epsilon_{3}$.
As $\alpha a_{1} \alpha=a_{2}$ and $\alpha a_{3} \alpha=a_{4}$, there is no subgroup isomorphic to ${\mathbb Z}_{2}^{2}$ inside $J$ such that the group generated by it and $\alpha$ is ${\mathbb Z}_{2}^{3}$. It follows that $J \cap J'$ cannot be ${\mathbb Z}_{2}^{2}$.
If $J \cap J'={\mathbb Z}_{2}$, then $J'$ must induces ${\mathbb Z}_{2}^{2}$, one generator being $a$ and the other being $b(z)=(z-\mu)/(z-1)$, this last one induced by some involution $\beta \in J' \setminus J$. It can be seen that
$\beta([y_{1}:y_{2}:y_{3}:y_{4}])=[y_{3}: \epsilon_{1} i  y_{4}: \epsilon_{2} \sqrt{1-\mu} \; y_{1}: \epsilon_{3} i \sqrt{1-\mu} \; y_{2}]$, where $\sqrt{1-\mu}$ is a fixed square root of $1-\mu$,
$\epsilon_{j} \in \{\pm1\}$ and $\epsilon_{2}=-\epsilon_{1}\epsilon_{3}$.
Similarly as for $\alpha$, one may see that $\beta a_{1} \beta=a_{3}$ and $\beta a_{2} \beta=a_{4}$. It can be checked that there is no involution of $J$ such that together $\alpha$ and $\beta$ generate ${\mathbb Z}_{2}^{3}$.

All the above produces a contradiction under the assumption that $H \neq H_{0}$.
\end{proof}

\section{Automorphisms of generalized Fermat manifolds}\label{Ssec:Aut}

\subsection{Upper bounds}
One direct consequence of the uniqueness of generalized Fermat groups is that one may obtain some information on the groups of automorphisms of generalized Fermat manifolds.

\begin{corollary}\label{coro-unico}
Let $d \geq 2$, $k \geq 2$, $n \geq d+1$ be integers and $(d;k,n) \notin \{(2;2,5), (2;4,3)\}$. Let $M$ be a generalized Fermat manifold of type $(d;k,n)$ and let $H$ be its unique
generalized Fermat group of type $(d;k,n)$. If $G_{0}$ is the ${\rm PGL}_{d+1}({\mathbb C})$-stabilizer of the $n+1$ branch hyperplanes of $M/H={\mathbb P}^{d}$, then $|{\rm Aut}(M)|= |G_0|k^{n}$ and, if the order of $G_{0}$ is relatively prime with $k$, then 
${\rm Aut}(M) \cong H \rtimes G_{0}$.
\end{corollary}
\begin{proof}
We know that $M$ admits a unique generalized Fermat group $H$ of type $(d;k,n)$. 
Let $\pi:M \to {\mathbb P}^{d}$ be a Galois branched covering, with $H$ as its desk group, and let $\{L_{1},\ldots, L_{n+1}\}$ be its set of branch hyperplanes. Let $G_{0}$ be the ${\rm PGL}_{d+1}({\mathbb C})$-stabilizer of these $n+1$ branch hyperplanes.
As $H$ is a normal subgroup of ${\rm Aut}(M)$, it follows the existence of a homomorphism $\theta:{\rm Aut}(M) \to G_{0}$, with kernel $H$.
As $M$ is a universal branched cover, every element $Q$ of $G_{0}$ lifts to a holomorphic automorphism $\widehat{Q}$ of $M$.
Then there is a short exact sequence
$1 \rightarrow H \rightarrow {\rm Aut}(M) \rightarrow G_0 \rightarrow 1.$
In particular,  $|{\rm Aut}(M)|= |G_0|k^{n}$. Also, by the Schur-Zassenhaus theorem \cite{Dummit}, in the case that the order of $G_{0}$ is relatively prime with $k$, then 
${\rm Aut}(M) \cong H \rtimes G_{0}$.
\end{proof}

\begin{corollary}
Let $d \geq 2$, $k \geq 2$. If
$G_{0}$ be a finite subgroup of ${\rm PGL}_{d+1}({\mathbb C})$, then there exists a generalized Fermat pair $(M,H)$ of type $(d;k,n)$, for some
$n \geq d+1$, such that ${\rm Aut}(M/H) \cong G_{0}$. In fact, for $|G_{0}| \leq d+1$ we may assume $n=d+1$ and, for $|G_{0}| \geq d+2$, we may assume $n=|G_{0}|-1$.
\end{corollary}
\begin{proof}
If $|G_{0}| \leq d+1$, then take $n=d+1$  and note that for the classical Fermat hypersurface $M_{n}^{k} \subset {\mathbb P}^{n}$ of degree $k$ one has that ${\rm Aut}(M_{n}^{k})/H$ contains the permutation group of $d+1$ letters.
Let us assume $|G_{0}| \geq d+2$.
The linear group $G_{0}$ induces a linear action on the space ${\mathbb P}^{d}_{hyper}$ of hyperplanes of ${\mathbb P}^{d}$. As $G_{0}$ is finite, we may find (generically) a point $q\in {\mathbb P}^{d}_{hyper}$ whose $G_{0}$-orbit is a generic set of points. Such an orbit determines a collection of $|G_{0}|$ lines in general position in ${\mathbb P}^{d}$. Let us observe that, by the generic choice, we may even assume the above set of points to have ${\rm PGL}_{d+1}({\mathbb C})$-stabilizer exactly $G_{0}$, so the same situation for our collection of hyperplanes.
Now, the results follow from Corollary \ref{coro-unico}.
\end{proof}

\subsection{Fixed points of elements of $H_{0}$}
Next, we consider the generalized Fermat pair $(M_{k}^{n}(\Lambda),H_{0})$ of type $(d;k,n)$.
This section describes the elements of $H_{0}$ having fixed points.

The diagonal presentation of the generators $\varphi_{j} \in H_{0}$, for $j=1,\ldots,n+1$, permits us to obtain the locus of fixed points of every element of $H_{0}$ (see Proposition \ref{observafijos} below). First, let us recall that each element of $H_{0}$ has the form $\varphi:=\varphi_{1}^{m_{1}}\cdots \varphi_{n+1}^{m_{n+1}} \in H_{0}$, where
 $m_{1},\ldots,m_{n+1} \in \{0,1,\ldots,k-1\}$.

Let us start with some definitions and notations.
For each tuple $(m_{1},\ldots,m_{n+1}) \in \{0,1,\ldots,k-1\}^{n+1}$ and each $l \in \{0,1,\ldots,k-1\}$, we set
{\small
$$L_{l}(m_{1},\ldots,m_{n+1}):=\{j \in \{1,\ldots, n+1\}: m_{j}=l\},$$
}
and the (possibly empty) algebraic sets
{\small
$$\widetilde{F}_{l}(m_{1},\ldots,m_{n+1})=\{[x_{1}:\cdots:x_{n+1}] \in {\mathbb P}^{n} : x_{i}=0, \;  \forall i \notin L_{l}(m_{1},\ldots,m_{n+1})\},$$
}
and
{\small
$$F_{l}(m_{1},\ldots,m_{n+1}):=\widetilde{F}_{l}(m_{1},\ldots,m_{n+1})\cap M^{k}_{n}(\Lambda).$$
}

\begin{proposition}\label{observafijos}
Let $d \geq 1$, $n \geq d+1$, $k \geq 2$, $\Lambda \in X_{n,d}$, $M^{k}_{n}(\Lambda)$, $H_{0} \cong {\mathbb Z}_{k}^{n}$ and $\varphi_{1}, \ldots, \varphi_{n+1} \in H_{0}$ as before. Then:
\begin{enumerate}[label=(\alph*),leftmargin=*,align=left]
 \item[(1)] $F_{l}(m_{1},\ldots,m_{n+1}) \neq \emptyset$  if and only if $\# L_{l}(m_{1},\ldots,m_{n+1}) \geq n+1-d$.
 \item[(2)] If  $F_{l}(m_{1},\ldots,m_{n+1}) \neq \emptyset$, then it is a  generalized Fermat manifold of dimension $\# L_{l}(m_{1},\ldots,m_{n+1})+d-n-1$.
 \item[(3)] Let $\varphi:=\varphi_{1}^{m_{1}}\cdots \varphi_{n+1}^{m_{n+1}} \in H_{0}$, different from the identity, where
 $m_{1},\ldots,m_{n+1} \in \{0,1,\ldots,k-1\}$. Then
its locus of fixed points is the disjoint union of the sets $F_{l}(m_{1},\ldots,m_{n+1})$, where $l \in \{0,1,\ldots,k-1\}$.
 In particular, the number of (non-empty) connected components of its locus of fixed points (if non-empty) equals the number of exponents $l$ appearing in $\varphi$ at least $n+1-d$ times.

\end{enumerate}
\end{proposition}
\begin{proof}
The locus of fixed points, in ${\mathbb P}^{n}$, of $\varphi$ is the disjoint union of the algebraic sets
{\small
$$\widetilde{F}_{l}(m_{1},\ldots,m_{n+1})=\{[x_{1}:\cdots:x_{n+1}] \in {\mathbb P}^{n}: x_{i}=0, \;  \forall i \notin L_{l}(m_{1},\ldots,m_{n+1})\},$$
}
where the union is taking on all those $l \in \{0,1,\ldots,k-1\}$ such that $\# L_{l}(m_{1},\ldots,m_{n+1}) \geq 1$.
Note that each $\widetilde{F}_{l}(m_{1},\ldots,m_{n+1})$ is:
(i) just a point if $\# L_{l}(m_{1},\ldots,m_{n+1})= 1$, and (ii) a projective space of dimension $\# L_{l}(m_{1},\ldots,m_{n+1}) -1$ if $\# L_{l}(m_{1},\ldots,m_{n+1})> 1$.
The locus of fixed points of $\varphi$ on $M^{k}_{n}(\Lambda)$ is then given as the disjoint union of the sets $F_{l}(m_{1},\ldots,m_{n+1})=\widetilde{F}_{l}(m_{1},\ldots,m_{n+1}) \cap M^{k}_{n}(\Lambda)$.
But on $M^{k}_{n}(\Lambda)$ we cannot have points $[x_{1}:\cdots:x_{n+1}]$ with at least $d+1$ coordinates equal to zero. This fact asserts that for $\# L_{l}(m_{1},\ldots,m_{n+1})\leq n-d$ one has that  $F_{l}(m_{1},\ldots,m_{n+1})=\emptyset$.
Also, for
$\# L_{l}(m_{1},\ldots,m_{n+1}) \geq n+1-d$, we obtain that $F_{l}(m_{1},\ldots,m_{n+1}) \neq \emptyset$ is a generalized Fermat manifold of dimension $\# L_{l}(m_{1},\ldots,m_{n+1})+d-n-1$.
\end{proof}

\begin{remark}\label{ejemplo3}
\begin{enumerate}[leftmargin=*,align=left]
\item If $k=2$, and $\varphi:=\varphi_{1}^{m_{1}}\cdots \varphi_{n+1}^{m_{n+1}} \in H_{0} \cong {\mathbb Z}_{2}^{n}$,
then $m_{1},\ldots,m_{n+1} \in \{0,1\}$. Proposition \ref{observafijos} asserts that $\varphi$ has no fixed points on $M^{2}_{n}(\Lambda)$ if and only if
$\#L_{0}(m_{1},\ldots,m_{n+1}),\#L_{1}(m_{1},\ldots,m_{n+1}) \leq n-d$.
As these two cardinalities must add to $n+1$, this is only possible for $n \geq 1+2d$.

\item
Let $k=2$, $d=1$, $n \geq 4$, and $C:=C_{n}^{2}(\Lambda)$, so $H_{0} \cong {\mathbb Z}_{2}^{n}$.
If $n \geq 5$ odd, then there is a (unique) subgroup $K \cong {\mathbb Z}_{2}^{n-1}$ of $H_{0}$ acting freely on $C$. For $n \geq 4$ even, the above is not true, but there are subgroups $K \cong {\mathbb Z}_{2}^{n-2}$ acting freely.

\item
Let $d \geq 2$, $k=p \geq 2$ be a prime integer,  $n \geq d+1$, and $M=M^{p}_{n}(\Lambda)$. Assume $K \cong {\mathbb Z}_{p}^{n-r}$ is a subgroup of $H_{0}$ acting freely on $M$. Let $F_{j} \subset M$, $j=1,\ldots, n+1$, be the locus of fixed points of the canonical generator $\varphi_{j}$. As $H_{0}$ is an abelian group, each $F_{j}$ is invariant under $K$ and acts freely on it.  Let $N=M/K$ (which is a compact complex manifold of dimension $d$) and $X_{j}=F_{j}/K$ (a connected complex submanifold of $N$). The $(n+1)$ connected sets $X_{j}$ are the locus of fixed points of the induced holomorphic automorphism by $\varphi_{j}$. As each two different $F_{i}$ and $F_{j}$ always intersect transversely, it follows that the same happens for $X_{i}$ and $X_{j}$. As the locus of fixed points of (finite) holomorphic automorphisms are smooth, it follows that different $X_{i}$ and $X_{j}$ are the fixed points of different cyclic groups of $A=H_{0}/K \cong {\mathbb Z}_{p}^{r}$. This in particular asserts that $n+1 \leq (p^{r}-1)/(p-1)$. So, for instance, the cases (i) $r=1$ and (ii) $r=2$ and $p=2$, are impossible (note that this is in contrast to the case $p=2$ and $d=1$, where these subgroups exist and are related to hyperelliptic Riemann surfaces).
\end{enumerate}
\end{remark}

\begin{example}
Let us take $n=k=3$ and $d=2$. In this case, $S^{3}_{3}$ is just the Fermat hypersurface $\{x_{1}^{3}+x_{2}^{3}+x_{3}^{3}+x_{4}^{3}=0\} \subset {\mathbb P}^{3}$. If
$\varphi=\varphi_{1}\varphi_{2}\varphi_{3}^{2}$, then
$(m_{1},m_{2},m_{3},m_{4})=(1,1,2,0)$ and
$L_{0}(1,1,2,0)=\{4\}, \; L_{1}(1,1,2,0)=\{1,2\}, \; L_{2}(1,1,2,0)=\{3\}$.
The locus of fixed points (in ${\mathbb P}^{3}$) of $\varphi$ is given by
{\small
$$\widetilde{F}_{0}(1,1,2,0) \cup \widetilde{F}_{1}(1,1,2,0) \cup \widetilde{F}_{2}(1,1,2,0)=$$
$$\{[0:0:0:1]\} \cup \{[x_{1}:x_{2}:0:0] \in {\mathbb P}^{3}\} \cup \{[0:0:1:0]\}.$$
}
As the cardinalities of $L_{0}(1,1,2,0)$ and $L_{2}(1,1,2,0)$ are equal to $1 \leq n-d$, these two do not introduce fixed points of $\varphi$ on $S_{3}^{3}$ (this can be seen also directly). The set $L_{1}(1,1,2,0)$ has cardinality $2 \geq n-d+1$, so it produces a zero-dimensional set of fixed points consisting of the three points $[1:-1:0]$, $[1:\omega_{6}:0]$ and $[1:\omega_{6}^{-1}:0]$, where $\omega_{6}=e^{\pi i/3}$.
\end{example}


\section{Examples}
\subsection{Jacobians of genus two surfaces}
In \cite{Kle70, Shi77}, it was observed that the desingularized Kummer surface of the Jacobian variety $JC$ of the genus two hyperelliptic curve
$C: y^{2}=(x-\alpha_{1})\cdots(x-\alpha_{6}),$
where $\alpha_{1},\ldots, \alpha_{6} \in {\mathbb C}$ are different points, is isomorphic to the following complete intersection
{\small
$$S=\left\{\begin{array}{rcl}
x_{1}^{2}+\cdots+x_{6}^{2}&=&0\\
\alpha_{1}x_{1}^{2}+\cdots+\alpha_{6}x_{6}^{2}&=&0\\
\alpha_{1}^{2}x_{1}^{2}+\cdots+\alpha_{6}^{2}x_{6}^{2}&=&0
\end{array}
\right\} \subset {\mathbb P}^{5}.
$$
}

The surface $S$ admits the group $H \cong {\mathbb Z}_{2}^{5}$ of automorphisms (generated by those elements that multiply each coordinate by $-1$), and the quotient $S/H$ is ${\mathbb P}^{2}$ branched at six lines in general position. This, in particular, asserts that $S$ is a generalized Fermat surface of type $(2;2,5)$. Below, we proceed to find an equation of it in the form  $S^{2}_{5}(\lambda_{1},\lambda_{2})$. If we set
$t_{1}=(\alpha_{1}-\alpha_{5})(\alpha_{1}-\alpha_{6})x_{1}^{2},
t_{2}=(\alpha_{2}-\alpha_{5})(\alpha_{2}-\alpha_{6})x_{2}^{2},
t_{3}=(\alpha_{3}-\alpha_{5})(\alpha_{3}-\alpha_{6})x_{3}^{2},$
then $P:S \to {\mathbb P}^{2}:[x_{1}:\cdots:x_{6}] \mapsto [t_{1}:t_{2}:t_{3}]$ defines a Galois branched covering with deck group $H$.
In this case, the six branched lines are
{\small
$$L_{1}=\{t_{1}=0\}, L_{2}=\{t_{2}=0\}, L_{3}=\{t_{3}=0\}, L_{4}=\{t_{1}+t_{2}+t_{3}=0\},$$
$$L_{5}=\{\frac{(\alpha_{1}-\alpha_{4})(\alpha_{3}-\alpha_{5})}{(\alpha_{1}-\alpha_{5})(\alpha_{3}-\alpha_{4})} t_{1}+
\frac{(\alpha_{2}-\alpha_{4})(\alpha_{3}-\alpha_{5})}{(\alpha_{2}-\alpha_{5})(\alpha_{3}-\alpha_{4})} t_{2}+t_{3}=0\},$$
$$L_{6}=\{\frac{(\alpha_{1}-\alpha_{4})(\alpha_{3}-\alpha_{6})}{(\alpha_{1}-\alpha_{6})(\alpha_{3}-\alpha_{4})} t_{1}+
\frac{(\alpha_{2}-\alpha_{4})(\alpha_{3}-\alpha_{6})}{(\alpha_{2}-\alpha_{6})(\alpha_{3}-\alpha_{4})} t_{2}+t_{3}=0\}.$$
}

It follows that $S \cong S_{5}^{2}(\lambda_{1},\lambda_{2})$, where
{\small
$$\lambda_{1}=\left(
\frac{(\alpha_{1}-\alpha_{4})(\alpha_{3}-\alpha_{5})}{(\alpha_{1}-\alpha_{5})(\alpha_{3}-\alpha_{4})},
\frac{(\alpha_{1}-\alpha_{4})(\alpha_{3}-\alpha_{6})}{(\alpha_{1}-\alpha_{6})(\alpha_{3}-\alpha_{4})}
\right),
$$
$$
\lambda_{2}=\left(
\frac{(\alpha_{2}-\alpha_{4})(\alpha_{3}-\alpha_{5})}{(\alpha_{2}-\alpha_{5})(\alpha_{3}-\alpha_{4})},
\frac{(\alpha_{2}-\alpha_{4})(\alpha_{3}-\alpha_{6})}{(\alpha_{2}-\alpha_{6})(\alpha_{3}-\alpha_{4})}
\right).
$$
}
\subsection{A connection between generalized Fermat surfaces and curves of the same type}
Let $(\lambda,\mu) \in X_{n,2}$, where $n \geq 4$, $\lambda=(\lambda_{1},\ldots,\lambda_{n-3}), \mu=(\mu_{1},\ldots,\mu_{n-3}) \in {\mathbb C}^{n-3}$.
Associated with this pair is
the generalized Fermat surface $S_{n}^{k}(\lambda,\mu)$, $k \geq 2$, together its
 (unique) generalized Fermat group $H=H_{0} \cong {\mathbb Z}_{k}^{n}$ and the Galois branched holomorphic map $\pi:S_{n}^{k}(\lambda,\mu) \to {\mathbb P}^{2}$ where $\pi[x_{1}:\cdots:x_{n+1}]=[t_{1}=x_{1}^{k}:t_{2}=x_{2}^{k}:t_{3}=x_{3}^{k}]$ (whose deck group is $H$). Its branch locus set is given by the union of the lines $L_{j}(\lambda,\mu)$, $j=1,\ldots,n+1$, as previously defined.

Let us consider a line $L=\{\rho_{1}t_{1}+\rho_{2}t_{2}+\rho_{3}t_{3}=0\} \subset {\mathbb P}^{2}$, $[\rho_{1}:\rho_{2}:\rho_{3}] \in {\mathbb P}^{2}$, such that the collection of $n+2$ lines $L_{1}(\lambda,\mu),\ldots,L_{n+1}(\lambda,\mu),L$ are in general possition.
On $L$ we have exactly $n+1$ intersection points with the branched lines $L_{1}(\lambda,\mu),\ldots,L_{n+1}(\lambda,\mu)$. These points are given by
{\small
$$p_{1}=[0:1], \; p_{2}=[1:0], \; p_{3}=\left[\frac{-\rho_{2}}{\rho_{1}}:1\right], \; p_{4}=\left[\frac{\rho_{3}-\rho_{2}}{\rho_{1}-\rho_{3}}:1\right],$$
$$p_{j}=\left[\frac{\mu_{j-4}\rho_{3}-\rho_{2}}{\rho_{1}-\lambda_{j-4}\rho_{3}}:1\right], j=5,\ldots, n+1.$$
}

The preimage $\pi^{-1}(L) \subset S_{n}^{k}(\lambda,\mu)$ provides the following curve
{\small
$$\pi^{-1}(L):=\left\{\begin{array}{rcl}
\rho_{1}x_{1}^{k}+\rho_{2}x_{2}^{k}+\rho_{3}x_{3}^{k}&=&0\\
x_{1}^{k}+x_{2}^{k}+x_{3}^{k}+x_{4}^{4}&=&0\\
\lambda_{1}x_{1}^{k}+\mu_{1}x_{2}^{k}+x_{3}^{k}+x_{4}^{k}&=&0\\
&\vdots& \\
\lambda_{n-3}x_{1}^{k}+\mu_{n-3}x_{2}^{k}+x_{n+1}^{k}&=&0\\
\end{array}
\right\} \subset {\mathbb P}^{n}.
$$
}

By taking
$y_{1}^{k}=\rho_{1} x_{1}^{k},\; y_{2}^{k}=\rho_{2}x_{2}^{k},\; y_{3}^{k}=\rho_{3}x_{3}^{k},\; y_{4}^{k}=\frac{\rho_{2}\rho_{3}}{\rho_{3}-\rho_{2}}x_{4}^{k},$
$y_{j}^{k}=\frac{\rho_{2}\rho_{3}}{\mu_{4-j}\rho_{3}-\rho_{2}}x_{j}^{k},$ for $j=5,\ldots,n+1,$
we observe that the above is isomorphic to the following generalized Fermat curve of type $(k,n)$
{\small
$$C_{n}^{k}(\eta_{L;\lambda,\mu})=\left\{\begin{array}{rcl}
y_{1}^{k}+y_{2}^{k}+y_{3}^{k}&=&0\\
\eta_{1}y_{1}^{k}+y_{2}^{k}+y_{4}^{k}&=&0\\
&\vdots& \\
\eta_{n-2}y_{1}^{k}+y_{2}^{k}+y_{n+1}^{k}&=&0\\
\end{array}
\right\} \subset {\mathbb P}^{n},
$$
}
where
$\eta_{L;\lambda,\mu}=(\eta_{1},\ldots,\eta_{n-2})$ and
{\small
$$\eta_{1}:=\frac{\rho_{2}(\rho_{3}-\rho_{1})}{\rho_{1}(\rho_{3}-\rho_{2})},\;
\eta_{2}:=\frac{\rho_{2}(\lambda_{1}\rho_{3}-\rho_{1})}{\rho_{1}(\mu_{1}\rho_{3}-\rho_{2})}, \ldots,
\eta_{n-2}:=\frac{\rho_{2}(\lambda_{n-3}\rho_{3}-\rho_{1})}{\rho_{1}(\mu_{n-3}\rho_{3}-\rho_{2})}.$$
}

By varying $L$ (and including the non-general ones as above), we obtain a two-dimensional family of such generalized Fermat curves for the fixed parameter $(\lambda,\mu)$. This process provides the $3$-fold
$X(\lambda,\mu)  \subset  {\mathbb P}^{2} \times {\mathbb P}^{n}$, defined by those tuples $([\rho_{1}:\rho_{2}:\rho_{3}], [x_{1}:\cdots:x_{n+1}])  \in  {\mathbb P}^{2} \times {\mathbb P}^{n}$ satisfying the following equations
{\small
$$\left\{\begin{array}{rcl}
y_{1}^{k}+y_{2}^{k}+y_{3}^{k}&=&0\\
\rho_{2}(\rho_{3}-\rho_{1}) y_{1}^{k}+\rho_{1}(\rho_{3}-\rho_{2}) y_{2}^{k}+ \rho_{1}(\rho_{3}-\rho_{1}) y_{4}^{k}&=&0\\
\rho_{2}(\lambda_{1}\rho_{3}-\rho_{1}) y_{1}^{k}+\rho_{1}(\mu_{1}\rho_{3}-\rho_{2}) y_{2}^{k}+ \rho_{1}(\mu_{1}\rho_{3}-\rho_{2}) y_{5}^{k}&=&0\\
&\vdots& \\
\rho_{2}(\lambda_{n-3}\rho_{3}-\rho_{1}) y_{1}^{k}+\rho_{1}(\mu_{n-3}\rho_{3}-\rho_{2}) y_{2}^{k}+ \rho_{1}(\mu_{n-3}\rho_{3}-\rho_{2}) y_{n+1}^{k}&=&0
\end{array}\right\},
$$
}
together the morphism
$\pi_{\lambda,\mu}:X(\lambda,\mu) \to {\mathbb P}^{2}: ([\rho_{1}:\rho_{2}:\rho_{3}],[:y_{1}:\cdots:y_{n+1}]) \mapsto [\rho_{1}:\rho_{2}:\rho_{3}],$
such that $C_{n}^{k}(\eta_{L;\lambda,\mu})\cong \pi_{\lambda,\mu}^{-1}([\rho_{1}:\rho_{2}:\rho_{3}])$.

The (possible singular) generalized Fermat curve $\pi_{\lambda,\mu}^{-1}(\rho)$ is invariant under the group $H \cong {\mathbb Z}_{k}^{n}=\langle b_{1},\ldots,b_{n}\rangle$  of automorphisms of $X(\lambda,\mu)$, where $b_{j}$ is amplification of $y_{j}$ by $e^{2\pi i/k}$, an such a restrictio to $\pi_{\lambda,\mu}^{-1}(\rho)$ is its generalized Fermat group of type $(k,n)$.

\begin{remark}
\begin{enumerate}[leftmargin=*,align=left]
\item If $\rho:=[\rho_{1}:\rho_{2}:\rho_{3}]\in {\mathbb P}^{2}$ corresponds to a line $L$ such that the collection of lines $L, L_{1}(\lambda,\mu), \ldots, L_{n+1}(\lambda,\mu)$ is not in general position, then $C_{n}^{k}(\eta_{L;\lambda,\mu})\cong \pi_{\lambda,\mu}^{-1}(\rho) \subset X(\lambda,\mu)$ is a singular curve (a singular generalized Fermat curve of type $(k,n)$).

\item The above provides a nice relation between the surface $S_{n}^{k}(\lambda,\mu)$ and the $3$-fold $X(\lambda,\mu)$. The last one parametrizes all the (singular) generalized Fermat curves of type $(k,n)$ inside the surface, each one being invariant under the generalized Fermat group $H$ of $S_{n}^{k}(\lambda,\mu)$.
\end{enumerate}
\end{remark}



\end{document}